\documentclass[12pt]{amsart}
\usepackage[utf8]{inputenc}
\usepackage{amsmath,amsthm,amssymb}
\usepackage{hyperref}
\usepackage{enumerate}
\usepackage{xcolor}
\usepackage{colortbl}
\usepackage{array}
\usepackage{gensymb} 
\usepackage{bm}
\usepackage[margin=1in]{geometry}
\usepackage{multicol}
\usepackage{comment}
\usepackage{graphicx}
\usepackage{tikz}

\newtheorem{definition}{Definition}

\newtheorem{theorem}{Theorem}
\newtheorem{corollary}{Corollary}

\newtheorem{proposition}{Proposition}

\newtheorem{example}{Example}

\begin{document}

\title{The Complete Positivity of Symmetric Tridiagonal and Pentadiagonal  Matrices}
\author[Lei Cao]{Lei Cao \textsuperscript{1,2}}
\author[Darian McLaren]{Darian McLaren \textsuperscript{3}}
\author[Sarah Plosker]{Sarah Plosker \textsuperscript{3}}

\thanks{\textsuperscript{1}School of Mathematics and Statistics, Shandong Normal University, Shandong, 250358, China}
\thanks{\textsuperscript{2}Department of Mathematics, Halmos College, Nova Southeastern University, FL 33314, USA}
\thanks{\textsuperscript{3}Department of Mathematics and Computer Science, Brandon University,
Brandon, MB R7A 6A9, Canada}



\maketitle

  \begin{abstract}
{We provide a decomposition that is sufficient in showing when a symmetric tridiagonal matrix $A$ is completely positive. Our decomposition can be applied to a wide range of matrices. We give alternate proofs for  a number of related results found in the literature in a simple, straightforward manner. We show that the cp-rank of any completely positive irreducible tridiagonal doubly stochastic matrix is equal to its rank. We then consider symmetric pentadiagonal matrices, proving some analogous results, and providing two different decompositions sufficient for complete positivity. We illustrate our  constructions with  a number of examples.}
\end{abstract}

\maketitle

\section{Preliminaries}

All matrices herein are real-valued, and in particular are entrywise non-negative. Let $A$ be an $n\times n$ symmetric tridiagonal matrix:
$$A=\begin{pmatrix}a_1&b_1&&&& \\ b_1 & a_2 & b_2 &&&\\ &\ddots&\ddots&\ddots&&& \\& &\ddots&\ddots&\ddots&& \\ &&&b_{n-3}&a_{n-2}&b_{n-2}& \\&&&&b_{n-2}&a_{n-1}&b_{n-1} \\&&&&&b_{n-1}&a_n \end{pmatrix}.$$
We are often interested in  the case where $A$ is also doubly stochastic, in which case we have
$a_{i}=1-b_{i-1}-b_{i}$ for all $i=1, 2,\ldots,n$, with the convention that $b_0=b_n=0$. It is easy to see that if a tridiagonal matrix is doubly stochastic, it must be symmetric, so the additional hypothesis of symmetry can be dropped in that case.

We are interested in positivity conditions for symmetric tridiagonal and pentadiagonal matrices. A stronger condition than positive semidefiniteness, known as complete positivity, has applications in a variety of areas of study, including block designs, maximin efficiency-robust tests,  modelling DNA evolution, and more
\cite[Chapter 2]{CP}, as well as recent use in mathematical optimization  and quantum information theory (see \cite{Nathaniel} and the references therein).

 With this motivation in mind, we study the positivity (in various forms) of  symmetric tridiagonal and pentadiagonal matrices, where we highlight the important case when the matrix is also doubly stochastic.
 Although it is NP-hard to determine if a given matrix is completely positive \cite{NPhard}, in Section~\ref{sec:td_CP} we provide a construction that is sufficient to show that a given symmetric tridiagonal 
 matrix is completely positive. We provide a number of examples illustrating the utility of this construction.
The literature on completely positive matrices often considers the cp-rank, or the factorization index, of a completely positive matrix,   which is the  minimal number of rank-one matrices in the   decomposition showing complete positivity;  e.g. Chapter 3 of \cite{CP} is devoted to this topic. We show that for irreducible tridiagonal doubly stochastic matrices, our decomposition is minimal.  It should be noted that it is known that  acyclic doubly non-negative matrices are completely
positive \cite{BH}, and this result has been generalized to bipartite doubly non-negative matrices \cite{BG}.   Our Proposition~\ref{PD_CP}  is an independent discovery of a special case of this result, using a simpler method of proof.

  As a natural extension of the tridiagonal case, we generalize many of our results to symmetric  pentadiagonal  matrices in Section~\ref{sec:penta}. While a construction analogous to that for the tridiagonal setting works in the pentadiagonal setting, we also  provide an alternate, more involved, construction that works in many cases when the original construction does not. A characterization used to determine the complete positivity of a matrix
with a particular graph is given in \cite{dogears}, with respect to the complete positivity of smaller matrices; they consider a particular non-crossing
cycle, which is the graph of a pentadiagonal matrix.

\section{Tridiagonal  matrices}\label{sec:td}
\subsection{Basic Properties of tridiagonal doubly stochastic matrices}\label{sec:td_basic}

Tridiagonal doubly stochastic matrices arise in the literature in a number of areas, in particular with respect to the study of Markov chains and in majorization theory. The facial structure of the set of all tridiagonal doubly stochastic matrices,  which is a subpolytope of the Birkhoff polytope  of $n\times n$ doubly stochastic matrices, is explored in \cite{Dahl}
with a connection to majorization. In \cite{Niezgoda}, the author develops relations involving sums of Jensen functionals to compare tuples of vectors; a tridiagonal doubly stochastic matrix is used to demonstrate their results.
In the study of mixing rates for Markov chains the assumption of symmetry in the transition matrix is sometimes seen, as in \cite{Boyd2009}. Other times, the Markov chain is assumed to be a path \cite{CihanAkar, Boyd2006} leading to a tridiagonal transition matrix. The properties of symmetric doubly stochastic matrices are explored in \cite{PereiraVali}, where majorization relations are given for the eigenvalues. Properties related to the facial structure of the polytope of tridiagonal doubly stochastic matrices can be found in \cite{FonsecaMarques, CostaFonseca2008, CostaFonseca}. In the former, alternating parity sequences are used to express the number of vertices of a given face, and in the latter, the number of $q$-faces of the polytope for arbitrary $n$ is determined for $q= 1,2,3$.

Many factorization techniques for tridiagonal matrices have been proposed in the literature; for example, \cite{pivot1, pivot2, pivot3} are concerned with $LXL^T$ factorizations and pivoting algorithms, where $L$ is unit lower triangular and $X$ is block diagonal with $1\times 1$ or $2\times 2$ blocks, factorization using parallel computers is studied in \cite{parallel}, a  factorization method for a symmetric singular value decomposition (Takagi factorization) is given for real and complex symmetric tridiagonal matrices in  \cite{SVD1} and \cite{SVD2}, respectively. Here, we present an algorithm to factor entrywise  non-negative, symmetric tridiagonal matrices in order to show complete positivity, discussed below.

One can ask under what conditions is a  tridiagonal doubly stochastic matrix $A$ positive semidefinite. It is known that a symmetric diagonally dominant matrix $A$ with   non-negative diagonal entries is positive semidefinite. Thus, in our case, if
\begin{equation}\label{eq1} b_{i-1}+b_{i}\leq 0.5 \end{equation} for all $i=1, 2,\ldots,n$,  with $b_0=b_{n}=0$, then $A$ is diagonally dominant, and hence $A$ is positive semidefinite. So \eqref{eq1} is sufficient for positive semidefiniteness of a tridiagonal doubly stochastic matrix. However, the following matrix is a tridiagonal doubly stochastic matrix that is positive semidefinite, showing that \eqref{eq1} is not necessary: 
\begin{eqnarray*}
\begin{pmatrix}
0.6&0.4&     0&   0\\
0.4& 13/30& 1/6& 0\\
0& 1/6& 13/30&   0.4\\
0&  0&  0.4 &0.6
\end{pmatrix}.   
\end{eqnarray*}

 %
We note that since tridiagonal doubly stochastic matrices are symmetric, their eigenvalues are real. Further, since the matrices are doubly stochastic, they always have $1$ as an eigenvalue (at least once), with corresponding eigenvector $\textbf{1}$ (the all-ones vector).
If $\lambda$ is an eigenvalue of a stochastic matrix, it is well-known that $\lambda\in \mathbb C$ such that $|\lambda|\leq 1$.   In our context, we note further that $-1\leq \lambda \leq 1$ (i.e., $\lambda\in \mathbb R$); this follows immediately from the assumption that our matrix is symmetric. 
%
%
%
%
In fact, we can say something stronger, as in the following proposition. 

\begin{proposition}\label{prop:eigval}
 Let $n\geq 2$. $\lambda$ is an eigenvalue of an $n \times n$ tridiagonal doubly stochastic matrix if and only if  $\lambda\in [-1,1]$.
\end{proposition}

\begin{proof}
 Suppose  $\lambda\in [-1,1]$ is arbitrary. The $2\times 2$ tridiagonal doubly stochastic matrix $A=\begin{pmatrix} a & b \\b & a\end{pmatrix}$
with $a+b=1$, $a\in [0,1]$, has eigenvalues $1$ and $2a-1$. So choose $a$ such that $2a-1=\lambda$, i.e.\ $a=(\lambda+1)/2$. Then $\lambda$ is an eigenvalue of the constructed matrix $A$. For $n>2$, note that we can construct an $n\times n$ tridiagonal doubly stochastic matrix via $A \oplus B$, where $B$ is an $(n-2) \times  (n-2)$ tridiagonal doubly stochastic matrix, and the constructed matrix $A \oplus B$ has $\lambda$ as an eigenvalue (if $v$ is an eigenvector corresponding to $\lambda$ for the matrix $A$, then $v\oplus \mathbf{0}_{n-2}$, where $\mathbf{0}_{n-2}$ is the $(n-2)$-dimensional zero vector,
is an eigenvector corresponding to $\lambda$ for $A\oplus B$). Thus one can construct a tridiagonal doubly stochastic matrix of arbitrary size having   the prescribed eigenvalue $\lambda$.

The converse follows from the discussion prior to this proposition: that the eigenvalues of a tridiagonal doubly stochastic matrix $A$ all lie in $[-1,1]$.
\end{proof}

\subsection{Complete Positivity}\label{sec:td_CP}
\begin{definition}\label{def:cp}
An $n\times n$ real matrix $A$ is \emph{completely positive} if it can be decomposed as $A=VV^T$, where $V$ is an $n\times k$ entrywise non-negative matrix, for some $k$.
\end{definition}

Equivalently, one can define $A$ to be completely positive provided $A=\sum_{i=1}^kv_iv_i^T$, where $v_i$ are entrywise non-negative vectors (namely, the columns of $V$).

Completely positive matrices are positive semidefinite and symmetric entrywise non-negative; such matrices are called \emph{doubly non-negative}. Doubly non-negative matrices are completely positive for $n\leq 4$, while doubly non-negative matrices that are not completely positive exist for all $n\geq 5$; see \cite{Berman88} and the references therein. In other words, the set of all completely positive matrices forms a strict subset of the set of all doubly non-negative matrices for $n\geq 5$.

We outline below a construction producing the completely positive decomposition $A=\sum_iv_iv_i^T$, which can be found by assuming that, since $A$ is tridiagonal, each $v_i$ should have only two nonzero entries (the $i$-th and $(i+1)$-th entries), and brute-force solving for these entries from the equation $A=VV^T$; these values can also be found somewhat indirectly, assuming our initial condition is zero, through a construction of  pairwise completely positive matrices in  \cite[Theorem 4]{Nathaniel} by taking both matrices to be $A$.

For a given $n\times n$ symmetric tridiagonal matrix $A$,  define the set $\{v_i\}_{i=0}^n$ of cardinality $n+1$, whose elements are $n$-dimensional vectors where the $j$-th component of $v_i$, denoted $(v_i)_j$, is recursively defined by
\begin{equation}\label{veqn}
(v_i)_j= \begin{cases}
     \sqrt{a_i-((v_{i-1})_i)^2} & j=i \\
     b_i/(v_i)_i & j=i+1 \\
     0  & otherwise
   \end{cases}
\end{equation}
 with initial condition $v_{0}=\begin{pmatrix} a_0&0& \dots & 0\end{pmatrix}^T$. This construction yields
   \begin{eqnarray*}
 v_{1}&=&\begin{pmatrix} \sqrt{a_1-a_0^2}& \frac{b_1}{\sqrt{a_1-a_0^2}}&0& \dots & 0\end{pmatrix}^T\\
 v_2&=&\begin{pmatrix} 0 & \sqrt{a_2-\frac{b_1^2}{a_1-a_0^2}}& \frac{b_2}{\sqrt{a_2-\frac{b_1^2}{a_1-a_0^2}}}& 0 & \dots & 0\end{pmatrix}^T\\
  v_3&=&\begin{pmatrix}0 &0&\sqrt{a_3-\frac{b_2^2}{a_2-\frac{b_1^2}{a_1-a_0^2}}}& \frac{b_3}{\sqrt{a_3-\frac{b_2^2}{a_2-\frac{b_1^2}{a_1-a_0^2}}}}& 0 & \dots & 0 \end{pmatrix}^T, \textnormal{ etc.}
 \end{eqnarray*}
 The constant $a_0$ must satisfy $a_0\geq 0$, however it is worth noting that certain values of $a_0$ (the most obvious case being $a_0^2 = a_1$) can lead to some of the $v_i$ vectors being ill-defined.
\begin{proposition}\label{tridiag_decomp}
Let $A$ be an $n\times n$ symmetric tridiagonal matrix and $a_0\geq 0$. If the $v_i$ as defined in Equation~\eqref{veqn} are well-defined, then $A=\sum_{i=0}^nv_iv_i^T$. If the entries for each $v_i$  are non-negative numbers, then $A$ is completely positive.
\end{proposition}

We note that if $A$ is entrywise non-negative, which includes the case of $A$ being doubly stochastic, and the entries of the $v_i$ are all real, then they are automatically non-negative.

\begin{proof}
Consider a symmetric tridiagonal matrix $A$ such that the vectors in Equation~\eqref{veqn} are well-defined. Let $V_i=v_iv_i^T$ for all $i=0,1,\dots,n$ and $\tilde{A}=\sum_{i=0}^nV_i$. We wish to show that $\tilde{A}=A$. From the definition of the $v_i$ given in Equation~\eqref{veqn}, each $V_i$ is tridiagonal with only up to four nonzero entries and so $\tilde{A}$ itself is tridiagonal. Now, consider a component $\tilde{a}_{j,j+1}$ of $\tilde{A}$, where $j=1,2,\dots,n-1$. The only $V_i$ that has a nonzero entry in the $(j,j+1)$-th component is $V_j$ as $v_j$ is the only vector with both the $j$ and $(j+1)$-th components being nonzero. The $(j,j+1)$-th component of $V_j$ is in fact $b_j$ and so $\tilde{a}_{j,j+1}=b_j$. By symmetry, we also have $\tilde{a}_{j+1,j}=b_{j}$. Now consider a component  on the diagonal of $\tilde{A}$:  $\tilde{a}_{jj}$, where $j=1,2,\dots,n$. The only $V_i$ that have nonzero entries in the $(j,j)$-th component are $V_j$ and $V_{j-1}$, with respective values $a_j-((v_{j-1})_j)^2$ and $((v_{j-1})_j)^2$. Clearly then $\tilde{a}_{jj}=a_j$ for all $j=1,2,\dots,n$. Therefore $A=\tilde{A}=\sum_{i=0}^nv_iv_i^T$; If the entries for each $v_i$ are non-negative numbers, then $A$ is completely positive.
\end{proof}

 There is a number of other related algorithms for finding the decomposition of a completely positive matrix. Method 5.3 of \cite{DD} equates showing that a circular matrix $A$ is completely positive with finding a solution to an optimization problem using a recurrence relation involving $a_i$ and $b_i$, not unlike Equation~\eqref{eq1}. The algorithm in \cite{Kaykobad}, which uses vertex-edge incidence matrices, applies to  non-negative diagonally dominant symmetric matrices. In the case of tridiagonal matrices, our algorithm is in fact more general than that of \cite{Kaykobad}, as our algorithm applies to non-negative positive semidefinite (tridiagonal) matrices. Our algorithm also gives the minimal completely positive decomposition (that is, it attains the cp-rank; see Corollary~\ref{cor:cp-rank}). 

\begin{example}\label{5x5}
Consider the $5\times 5$ case, which is the first (in terms of smallest dimension) non-trivial case.  For the matrices
\begin{eqnarray*}
A=\begin{pmatrix}
3/4&1/4&0 & 0 & 0\\
1/4&1/2&1/4& 0 & 0\\
0 & 1/4 & 1/2& 1/4 & 0\\
0 & 0 & 1/4& 1/2 & 1/4\\
0 & 0 & 0 & 1/4 & 3/4
\end{pmatrix} \quad \textnormal{and}\quad B= \begin{pmatrix}
7/9&2/9&0 & 0 & 0\\
2/9&5/9&2/9 & 0 & 0\\
0 & 2/9 & 7/9& 0 & 0\\
0 & 0 & 0& 8/9 & 1/9\\
0 & 0 & 0 & 1/9 & 8/9
\end{pmatrix}
\end{eqnarray*}
our construction with $a_0=0$ gives $A=VV^T$ and $B=WW^T$ where
\begin{eqnarray*}
V=\begin{pmatrix}
\frac{1}{2}\sqrt{3}&0&0 & 0 & 0 \\
\frac{1}{2\sqrt{3}}&\frac{1}{2}\sqrt{\frac{5}{3}}&0 & 0 & 0 \\
0 & \frac{1}{2}\sqrt{\frac{3}{5}}& \frac{1}{2}\sqrt{\frac{7}{5}}& 0 & 0\\
0 & 0 & \frac{1}{2}\sqrt{\frac{5}{7}}& \frac{3}{2\sqrt{7}} & 0\\
0 & 0 & 0 & \frac{1}{6}\sqrt{7} & \frac{1}{3}\sqrt{5}
\end{pmatrix} \quad \textnormal{and}\quad
W=\begin{pmatrix}
\frac{1}{3}\sqrt{7}&0&0 & 0 & 0\\
\frac{2}{3\sqrt{7}}&\frac{1}{3}\sqrt{\frac{31}{7}}&0 & 0 & 0\\
0 & \frac{2}{3}\sqrt{\frac{7}{31}}& \sqrt{\frac{21}{31}}& 0 & 0\\
0 & 0 & 0& \frac{2}{3}\sqrt{2} & 0\\
0 & 0 & 0 & \frac{1}{6\sqrt{2}} & \frac{1}{2}\sqrt{\frac{7}{2}}
\end{pmatrix}.
\end{eqnarray*}
Therefore the matrices $A$ and $B$ are completely positive. Note that $V$ and $W$ should be $5\times 6$ matrices; however, the selection of $a_0=0$ forces $v_0$ to be the zero vector and as such the first column of both $V$ and $W$ is all zeroes, so can be omitted. For this reason, choosing $a_0=0$ often leads to a much simpler decomposition.

It is important to emphasise here that a decomposition proving that a matrix $A$ is completely positive is in general not unique. In particular, the choice of $a_0$ can lead to different decompositions, assuming they are still well-defined.
For example, if we had instead chosen $a_0=3/4$, the matrix
\begin{eqnarray*}
\tilde{W}=\begin{pmatrix}
\frac{3}{4}&\frac{1}{12}\sqrt{31}&0&0 & 0 & 0\\
0&\frac{8}{3\sqrt{31}}&\frac{1}{3}\sqrt{\frac{91}{31}}&0 & 0 & 0\\
0&0 & \frac{2}{3}\sqrt{\frac{31}{91}}& \sqrt{\frac{57}{91}}& 0&0 \\
0&0 & 0 & 0& \frac{2}{3}\sqrt{2} & 0\\
0&0 & 0 & 0 & \frac{1}{6\sqrt{2}} & \frac{1}{2}\sqrt{\frac{7}{2}}
\end{pmatrix}
\end{eqnarray*}
works in the decomposition of $B$.
\end{example}

If the given matrix is in block form but our decomposition does not work, we may employ the technique illustrated in the example below: treating each block separately.
\begin{example}\label{triBlockEx}
Consider the matrix
\begin{eqnarray*}
C=\begin{pmatrix}
1 & 0 & 0 & 0 & 0\\
0 & 1/2&1/2&0 & 0\\
0 & 1/2&1/2&0 & 0\\
0 & 0 & 0 & 1/2 & 1/2\\
0 & 0 & 0 & 1/2 & 1/2
\end{pmatrix}.
\end{eqnarray*}
Since we have $b_1=0$ this gives $(v_1)_2=0$. Therefore we also have $(v_3)_3=\sqrt{a_3-\frac{b_2^2}{a_2}}=\sqrt{1/2-1/2}=0$. Hence, regardless of our choice of $a_0$ the component $(v_3)_4$ is never well-defined. To get around this fact consider $C$ as the block matrix
\begin{eqnarray*}
C=\begin{pmatrix}
C_1 & 0_{3,2}\\
0_{2,3} & C_2
\end{pmatrix}.
\end{eqnarray*}
where $0_{n,m}$ denotes the $n\times m$ all-zeros matrix and
\begin{eqnarray*}
C_1=\begin{pmatrix}
1 & 0 & 0\\
0&1/2&1/2\\
0&1/2&1/2
\end{pmatrix} \quad \textnormal{and}\quad
C_2=\begin{pmatrix}
1/2&1/2\\
1/2&1/2\\
\end{pmatrix}.
\end{eqnarray*}
The matrices $C_1$ and $C_2$ on the other hand we have no issues with decomposing. Choosing $a_0=0$ we obtain
\begin{eqnarray*}
V_1=\begin{pmatrix}
1&0&0\\
0&\frac{1}{\sqrt{2}} & 0\\
0&\frac{1}{\sqrt{2}}&0
\end{pmatrix} \quad \textnormal{and}\quad
V_2=\begin{pmatrix}
\frac{1}{\sqrt{2}} & 0\\
\frac{1}{\sqrt{2}}&0
\end{pmatrix}
\end{eqnarray*}
where $C_1=V_1V_1^T$ and $C_2=V_2V_2^T$. Therefore
\begin{eqnarray*}
V=\begin{pmatrix}
V_1 & 0_{3,2}\\
0_{2,3} & V_2
\end{pmatrix}
=\begin{pmatrix}
1&0&0&0&0\\
0&\frac{1}{\sqrt{2}} & 0&0&0\\
0&\frac{1}{\sqrt{2}}&0&0&0\\
0&0&0&\frac{1}{\sqrt{2}} & 0\\
0&0&0&\frac{1}{\sqrt{2}}&0
\end{pmatrix}
\end{eqnarray*}
where $C=VV^T$ and hence $C$ is completely positive.
\end{example}

The decomposition given by Equation~\eqref{veqn} leads to the following result, which includes tridiagonal doubly stochastic positive definite matrices. Corollary~4.11 of \cite{BH} (doubly non-negative tridiagonal matrices are completely positive) encompasses this result, but its proof relies on deletion of a leaf from a connected acyclic graph;  our method of proof is more direct, does not involve graph theory,  and relies solely on Equation~\eqref{veqn} and observing the leading principal minors of the matrix.

\begin{proposition}\label{PD_CP}
If $A$ is a symmetric tridiagonal 
positive definite entrywise non-negative matrix, then $A$ is  completely positive.
\end{proposition}

\begin{proof}
 Sylvester's criterion tells us that for a real symmetric matrix $A$, positive definiteness is equivalent to all leading principal minors of $A$ being positive.

Note that all square roots in the denominators of the entries in Equation~(\ref{veqn}) being well-defined with $a_0=0$ imply that all leading principal minors are positive. Indeed,
taking $a_0=0$ in the construction of Equation~\eqref{veqn}, we find the following.
For  $v_1$ to be well-defined, we have $a_1> 0$, which is the $1\times1$ leading principal minor.

For $v_2$ to be well-defined, we have $\displaystyle a_2-\frac{b_1^2}{a_1}> 0$, which is equivalent to $\displaystyle a_1a_2-b_1^2> 0$, which is the $2\times 2$ leading principal minor.

For  $v_3$ to be well-defined, we have $\displaystyle a_3-\frac{b_2^2}{a_2-\frac{b_1^2}{a_1}}> 0$, which is equivalent to $\displaystyle a_1a_2a_3-a_3b_1^2-a_1b_2^2> 0$, which is the $3\times 3$ leading principal minor.

Continuing in this manner, the result follows.
\end{proof}


 The construction of \cite{HwangPo} provides a method to construct a symmetric doubly stochastic matrix with prescribed eigenvalues, however it becomes trivial (most eigenvalues need to be 1) if we further restrict to tridiagonal matrices. The positive definiteness of tridiagonal matrices was considered in \cite{revisited}, which also makes use of chain sequences:

\begin{definition}Let $k$ be a positive integer. A sequence $\{\alpha_k\}_{k>0}$ is called a  \emph{(positive) chain sequence} if there exists a parameter sequence $\{g_k\}_{k>0}$ such that
$$\alpha_k=g_k(1-g_{k-1}),$$ with $0 \leq g_0 < 1$ and $0 < g_k < 1$, for $k > 0.$
\end{definition}

 In Theorem~\ref{SylWW} below, we take advantage of the doubly stochastic structure of the matrix $A$ to efficiently determine the existence of a unique   tridiagonal doubly stochastic matrix with prescribed diagonal entries, rather than prescribed eigenvalues. We then use Sylvester's criterion, the Wall-Wetzel Theorem (\cite[Theorem 3.3]{revisited}, \cite{WW}), and our construction presented in Equation~\eqref{veqn} to show the equivalence of positive definiteness, complete positivity, and a set of  inequalities involving the principal minors of the given matrix. 

\begin{theorem}\label{SylWW}

Let $A$ be a tridiagonal doubly stochastic matrix of the form
$$A=\begin{pmatrix}
a_1 & b_1 &&&&&\\
b_1 & a_2 & b_2 &&&& \\
&\ddots&\ddots&\ddots&&&\\
&&\ddots&\ddots&\ddots&&\\
&&&\ddots&\ddots&\ddots&\\
&&&&b_{n-2}&a_{n-1}&b_{n-1}\\
&&&&&b_{n-1}&a_n
\end{pmatrix}.$$

Denote $\mathbf{a}=(a_1, a_2, \ldots, a_{n-1})\in \mathbb{R}^{n-1}.$ Then $A$ is uniquely determined by the vector $\mathbf{a}$ if and only if $\mathbf{a}$ fulfills Equations~\eqref{eq:a_ipos} and~\eqref{eq:a_i} given below:

\begin{equation}\label{eq:a_ipos} 0\leq a_i \leq 1  \end{equation}  for all $i=1,2,\ldots,n-1.$
 \begin{eqnarray} \label{eq:a_i}
0<(a_1+a_3+\ldots+a_{2k-1})-(a_2+a_4+a_6+\ldots+a_{2k})<1
\end{eqnarray}   for all $k=1,2,\ldots,\left\lfloor\frac{n}{2}\right\rfloor.$

Denote $\mathbf{b}=(b_1, b_2, \ldots, b_{n-1})\in \mathbb{R}^{n-1}.$ Then $A$ is uniquely determined by the vector $\mathbf{b}$ if and only if $\mathbf{b}$ fulfills Equation~\eqref{b_i} below:

\begin{equation} \label{b_i}
b_{i}=\begin{cases}
      1-a_1+a_2-a_3+\ldots-a_{i}  & \quad \text{if } i \text{ is odd}  \\
       a_1-a_2+a_3-\ldots-a_{i}  & \quad \text{if }i  \text{ is even}
\end{cases}
\end{equation}
for all $i=1,2,3,\ldots, n-1.$

Moreover, the following are equivalent:
\begin{enumerate}[(a)]
\item $A$ is positive definite
\item $A$ is completely positive
\item  \begin{equation}\nonumber a_i > b^2_{i-1}\frac{\det (A_{i-2})}{\det(A_{i-1})}\end{equation}
for all $i=2,3,\ldots,n,$ where $A_i$ is the upper left $i\times i$ submatrix of $A$ (with $A_n=A$).
\end{enumerate}
\end{theorem}

\begin{proof}
An $n\times n$ symmetric tridiagonal doubly stochastic matrix is uniquely determined by its diagonal entries $\mathbf{a}=(a_1, a_2, \ldots, a_n)$:  A given vector $\mathbf{a}\in \mathbb{R}^n$ determines a tridiagonal doubly stochastic matrix if and only if both Equation~\eqref{eq:a_ipos}  and \eqref{eq:a_i} hold. Indeed, since $b_{i}\geq 0$ for all $i=1,2,\ldots,n-1,$ it can easily be seen that  Equation~\eqref{b_i} is equivalent to Equation~\eqref{eq:a_i}, while Equation~\eqref{eq:a_ipos} is a necessary condition for the matrix to be doubly stochastic.

Next, we note that the determinant of $A_n$ satisfies a three-term recurrence relation
\begin{equation}\label{3tr}\det(A_i)=a_i \det (A_{i-1})-b^2_{i-1}\det (A_{i-2})\end{equation}
for all $i=2,3,\ldots,n$ where $\det(A_0)$ is defined to be $1.$
Sylvester's criterion states that $A$ is positive definite if and only if  $\det(A_i)>0$ for all $i$, from which it follows that \eqref{3tr} implies that
\begin{equation}\label{eq2}a_i > b^2_{i-1}\frac{\det (A_{i-2})}{\det(A_{i-1})}\end{equation}
for $i=2,3,\ldots,n$ if and only if $A$ is positive definite.

We now consider the sequence $$\alpha_k=\frac{b_k^2}{a_ka_{k+1}},$$ for all $k=1, 2, \dots, n$, which, assuming entrywise non-negativity of $A$,  is a chain sequence with  $g_0=0$ and  and  $$g_k=\frac{b_k^2}{a_ka_{k+1}}\cdot \frac{1}{1-g_{k-1}}$$ for $k=1,2,\ldots, n$. Note that $0<g_k<1$ for $k=1, 2, \dots, n$. Indeed, it is clear that $g_1, \dots, g_n$ are positive, and to see they are less than one, we note that the inequality $g_1<1$ is equivalent to $\displaystyle \operatorname{det}\begin{pmatrix}
   a_1&b_1\\b_1&a_2
\end{pmatrix}=a_1a_2-b_1^2>0$, the inequality $g_2<1$ is equivalent to $\displaystyle \operatorname{det}\begin{pmatrix}
   a_1&b_1 & 0\\b_1&a_2&b_2\\ 0 & b_2 & a_3
\end{pmatrix}=a_1a_2a_3-a_1b_2^2-a_3b_1^2>0$, and in general $g_k<1$ is equivalent to $\displaystyle \operatorname{det}\begin{pmatrix}a_1&b_1&& \\ b_1 & a_2 & b_2 &\\ &\ddots&\ddots&\ddots&& \\ &&b_{k-2}&a_{k-1}&b_{k-1} \\&&&b_{k-1}&a_k \end{pmatrix}>0$

In our algorithm to construct vectors $v_i$ with components given by Equation~\eqref{veqn}, we require all radicands to be positive. Letting $a_0=0,$ we have the radicand in $(v_k)_k$ is $$a_k-\frac{b_{k-1}^2}{a_{k-1}(1-g_{k-2})}=a_3(1-\frac{b_2^2}{a_2(1-g_1)})=a_k(1-g_{k-1})$$
for $k=2,3,\ldots, n.$
The Wall-Wetzel Theorem \cite{revisited, WW}  states that a real symmetric tridiagonal matrix with positive diagonal entries is positive definite if and only if $\displaystyle \left\{\frac{b_k^2}{a_{k}a_{k+1}}\right\}_{k=1}^{n-1}$ is a chain sequence. Note that a matrix cannot be positive definite if there is at least one diagonal entry equal to 0. This 
implies that   the radicand in $(v_k)_k$ is positive for all $k$. Our construction therefore provides well-defined vectors $v_k$, and therefore shows that $A$ is completely positive. Thus we have (a) $\Rightarrow$
 (b); that is, positive definiteness implies complete positivity.

It is clear that (b) implies (a). Thus the equivalence of (a), (b), and (c) follows.
\end{proof}

Note that $A_i$ only contains $a_1, a_2,\ldots, a_{i}$ for $i=1,2, \ldots, n,$ so \eqref{eq2} in the above proof gives a condition (namely, a lower bound) that each $a_i$ has to satisfy and only depends on the entries $a_1, a_2, \ldots, a_{i-1}$.
For example,
$$a_4 > b_3^2 \frac{\det(A_2)}{\det(A_3)}=(1-a_1+a_2-a_3)^2 \frac{\det\begin{pmatrix}a_1 & 1-a_1 \\ 1-a_1 & a_2 \end{pmatrix}}{\det\begin{pmatrix}a_1 & 1-a_1&0 \\ 1-a_1 & a_2 & a_1-a_2 \\ 0& a_1-a_2 & a_3 \end{pmatrix}}$$
which only contains $a_1,a_2,a_3$ on the righthand side.


\begin{definition}\label{irred}
Let $A$ be an $n\times n$ matrix. $A$ is said to be \emph{reducible} if it can be transformed via row and column permutations to a block upper triangular matrix, with block sizes $<n$. $A$ is \emph{irreducible} if it is not reducible.
\end{definition}

Note that in the context of tridiagonal doubly stochastic matrices, irreducibility is equivalent to $b_i>0$ for all $i=1, \dots, n$, i.e.\ that $A$ cannot be written as the direct sum of smaller tridiagonal doubly stochastic matrices. When considering whether or not a tridiagonal doubly stochastic matrix $A$ is completely positive, we may assume $A$ is irreducible. Indeed, if $A$ were a direct sum of smaller doubly stochastic matrices---implying that some $b_i=0$---we could consider these smaller doubly stochastic matrices separately. The  $V$ corresponding to $A$ in the decomposition would have the same direct sum structure: it would be the direct sum of the $V$'s corresponding to the smaller doubly stochastic matrices. If some $b_i=0,$ then the corresponding $v_i$ only has one nonzero element.  $B$ in Example 1 is a direct sum of two doubly stochastic matrices and $C$ in Example 2 is a direct sum of three doubly stochastic matrices.

It is clear that if a matrix $A$ is completely positive, then it is automatically positive semidefinite.
Taussky's theorem \cite[Theorem II]{Taussky} allows us to use the eigenvalues of a given tridiagonal doubly stochastic matrix to characterize a partial  converse statement.

\begin{theorem}(Taussky's Theorem) Let $A$ be an $n\times n$  irreducible  matrix. An eigenvalue of $A$ cannot lie on the boundary of a Gershgorin disk unless it lies on the boundary of every Gershgorin disk.
\end{theorem}

Equivalently, Taussky's theorem states that if $A$ is an irreducible, diagonally dominant matrix with at least one inequality of the diagonal dominance being strict (in the context of tridiagonal doubly stochastic matrices, this means that \eqref{eq1} holds with strict inequality for at least one $i$), then $A$ is nonsingular. This is in fact the original formulation of the theorem in \cite{Taussky}.

Recall  that if a matrix $A$ is  symmetric, diagonally dominant, and  all its diagonal entries are non-negative, then $A$ is positive semidefinite.

\begin{proposition}\label{PSD_PD}
Let $A$ be an $n\times n$ irreducible tridiagonal doubly stochastic matrix. If $n\geq 3$ and $A$ is diagonally dominant, then $A$ is positive definite or, equivalently, $A$ is nonsingular.
\end{proposition}

\begin{proof} If $A$ is diagonally dominant, then $b_i+b_{i+1}\leq 0.5$ for all $i=0,1,2,\ldots,n$ with $b_0=b_{n}=0.$ Suppose $0$ is an eigenvalue of $A,$ then by Gorshgorin circle theorem, there exists some $i$ such that $b_i+b_{i+1}=0.5$ meaning that 0 is an eigenvalue on the boundary of a disk, so by Taussky's Theorem, every disk must have boundary at 0; that is, $b_i+b_{i+1}=0.5$ for all $i$. We have $b_1=0.5,$ which makes $b_2=0,$ which contradicts with the assumption that $A$ is irreducible.
\end{proof}

Note that the tridiagonal doubly stochastic matrix $\begin{pmatrix} 1/2 & 1/2\\ 1/2& 1/2\end{pmatrix}$ is the only $2\times 2$ tridiagonal doubly stochastic matrix that is positive semidefinite, without being positive definite (that is, it is the only $2\times 2$ positive semidefinite tridiagonal doubly stochastic matrix with zero as an eigenvalue). One can see this from \eqref{eq1} and the equivalent formulation of Taussky's theorem. One can verify that it is completely positive with $V=1/\sqrt{2}(1\,\, 1)^T$.

A number of corollaries follow from Proposition~\ref{PSD_PD}.

\begin{corollary}\label{cor:1/2}
Let $A$ be an $n\times n$  tridiagonal doubly stochastic matrix with $n\geq 3$. If $A$ is diagonally dominant, with zero as an eigenvalue, then $A$ must be reducible with at least one block of the form $\begin{pmatrix} 1/2 & 1/2\\ 1/2& 1/2\end{pmatrix}$.
\end{corollary}

The following Corollary is a more general statement than our previous Proposition~\ref{PD_CP}. This result appears to be known (e.g.\ it is mentioned in \cite[Section 3]{Dahl}), yet we are unaware of a proof in the literature. Given some subtleties in, and the length of, the proof, we have provided the details herein, which culminate in the corollary below. We note that \cite[Example 2]{CPgraphs} states that all tridiagonal doubly stochastic matrices are completely positive, which is not true in general without the assumption of diagonal dominance.

\begin{corollary}\label{TridiagDS_CP=PSD}
Let $A$ be an $n\times n$  tridiagonal doubly stochastic matrix. If $A$ is diagonally dominant, then $A$ is completely positive.
\end{corollary}

\begin{proof}
From the discussion following Definition~\ref{irred}, we can assume that $A$ is irreducible.

Proposition~\ref{PSD_PD} in fact holds for $n\geq 2$ except for the special case of the $2\times 2$ matrix $A=\begin{pmatrix} 1/2 & 1/2\\ 1/2& 1/2\end{pmatrix}$. However, this matrix is  completely positive with $V=1/\sqrt{2}(1\,\, 1)^T$. In all other cases, it  follows from Proposition~\ref{PSD_PD} that $A$ is positive definite and hence $A$ is completely positive due to Theorem \ref{SylWW}.

\end{proof}

Proposition 3.2 of \cite{CP} states that the
cp-rank of a matrix $A$ (that is, the minimal number $k$ in Definition~\ref{def:cp})  is greater than or equal to the rank of $A$. 
We show that the cp-rank of any tridiagonal doubly stochastic matrix is the same as the rank, or equivalently, Equation~(\ref{veqn}) provides a way to construct a minimal rank-one decomposition.

\begin{corollary}\label{cor:cp-rank}
The cp-rank of any positive semi-definite irreducible tridiagonal doubly stochastic matrix is equal to its rank.
\end{corollary}

\begin{proof}
Let $A$ be an $n\times n$ irreducible tridiagonal doubly stochastic matrix.
 We can let $a_0=0$ in the construction of the $v_i$  in Equation~\eqref{veqn}  as long as the matrix $A$ is positive definite by the proof of Proposition~\ref{PD_CP}. Therefore, the number of nonzero summands in Equation~(\ref{veqn}) is at most $n$.

An irreducible tridiagonal doubly stochastic matrix is singular if and only if it is the $2\times 2$ matrix with all entries the same (equal to $1/2$). Indeed, singular means that the rows are linearly dependent, but this is impossible for $n\geq 3$ since irreducibility of a  tridiagonal doubly stochastic matrix is  equivalent to $b_i>0$ for all $i=1, \dots, n$ (and thus the rank of such a matrix is $n$). As previously stated, in the case of the $2\times 2$ all-$1/2$ matrix (which has rank 1), our decomposition gives a $2\times 1$ matrix $V=1/\sqrt{2}(1\,\, 1)^T$ (equivalently, a single vector $v=1/\sqrt2(1\,\, 1)^T$). For $n\geq 3$, an irreducible tridiagonal doubly stochastic matrix has rank $n$, and our decomposition gives $n$ vectors $v_1, \dots, v_n$.
\end{proof}

\section{Symmetric pentadiagonal   matrices}\label{sec:penta}

Let $A$ be an $n\times n$ symmetric pentadiagonal  matrix:
$$A=\begin{pmatrix}a_1&b_1&c_1&&&& \\ b_1 & a_2 & b_2&c_2 &&&\\ c_1& b_2 & a_3 & b_3&c_3 &&\\ &\ddots&\ddots&\ddots&&& \\& &\ddots&\ddots&\ddots&\ddots& \\ &&c_{n-4}&b_{n-3}&a_{n-2}&b_{n-2}&c_{n-2} \\
&&&c_{n-3}&b_{n-2}&a_{n-1}&b_{n-1} \\&&&&c_{n-2}&b_{n-1}&a_{n} \end{pmatrix}. $$
We are interested again in the setting where $A$ is entrywise non-negative. If $A$ is also doubly stochastic, we have
 $a_{i}=1-(b_{i-1}+b_{i}+c_{i-2}+c_i)$ for $i=1,2,\ldots,n$, where $b_k=0$ for $k\leq 0$ or $k\geq n$ and $c_\ell=0$ for $\ell\leq 0$ or $\ell\geq n-1$.

 Unlike in the tridiagonal matrix setting, the property of being doubly stochastic does not immediately imply symmetry, and thus we assume as a hypothesis this extra condition.

\subsection{Basic Properties of symmetric pentadiagonal doubly stochastic matrices}
Many of the arguments from Section~\ref{sec:td_basic} carry through into the pentadiagonal setting.
As in the tridiagonal case, the eigenvalues of a symmetric pentadiagonal doubly stochastic matrix  are bounded between $-1$ and $1$.


Again, as in the case of tridiagonal doubly stochastic matrices (Proposition~\ref{prop:eigval}), any value in $[-1,1]$ can be realized as an eigenvalue of an $n\times n$ symmetric pentadiagonal doubly stochastic matrix. That is, the statement of Proposition~\ref{prop:eigval} reads the same when ``tridiagonal'' is replaced with ``symmetric pentadiagonal''. One can, if desired, use the bonafide pentadiagonal matrix\[A=\begin{pmatrix}a& b & b\\b & a & b\\b&b&a
\end{pmatrix}\] for the $n\geq 3$ cases.


\subsection{Complete Positivity}
We now provide a construction similar to that for tridiagonal doubly stochastic matrices, to provide a sufficient condition for when a symmetric pentadiagonal doubly stochastic matrix $A$ is completely positive. Define the set $\{v_i\}_{i=-1}^n$ of cardinality $n+2$, whose elements are $n$-dimensional vectors where the $j$-th component of $v_i$, denoted $(v_i)_j$ (where $j=1, \dots, n$), is recursively defined by
\begin{equation}\label{veqn_pent}
(v_i)_j= \begin{cases}
     \sqrt{a_i-[((v_{i-1})_i)^2+((v_{i-2})_i)^2]} & j=i \\
     \frac{b_i-(v_{i-1})_j(v_{i-1})_{j-1}}{(v_i)_i} & j=i+1 \\
     c_i/(v_i)_i & j=i+2 \\
     0  & otherwise
   \end{cases}
\end{equation}
  with initial conditions $v_{-1}=\begin{pmatrix} a_{-1}&0& \dots & 0\end{pmatrix}^T$ and $v_{0}=\begin{pmatrix} a_{0}&b_{0}&0& \dots & 0\end{pmatrix}^T$. This construction yields
   \begin{eqnarray*}
 v_{1}&=&\begin{pmatrix} \sqrt{a_1-(a_0^2+a_{-1}^2)}& \frac{b_1-b_0a_0}{\sqrt{a_1-(a_0^2+a_{-1}^2)}}& \frac{c_1}{\sqrt{a_1-(a_0^2+a_{-1}^2)}}&0& \dots & 0\end{pmatrix}^T\\
 v_2&=&\begin{pmatrix} 0 & \sqrt{a_2-\left(\frac{(b_1-b_0a_0)^2}{a_1-(a_0^2+a_{-1}^2)}+b_0^2\right)}& \frac{b_2-\frac{c_1(b_1-b_0a_0)}{a_1-(a_0^2+a_{-1}^2)}}{\sqrt{a_2-\left(\frac{(b_1-b_0a_0)^2}{a_1-(a_0^2+a_{-1}^2)}+b_0^2\right)}}& \frac{c_2}{\sqrt{a_2-\left(\frac{(b_1-b_0a_0)^2}{a_1-(a_0^2+a_{-1}^2)}+b_0^2\right)}}& 0 & \dots & 0\end{pmatrix}^T\\
  &\textnormal{ etc.}&
 \end{eqnarray*}
 Similar to the tridiagonal case, the constants $a_{-1},a_0,$ and $b_0$ are taken to be non-negative numbers with the caveat that there is some collection of initial values that leads to the decomposition being ill-defined. In fact, let $a_{-1}=b_0=0$ and $v_{-1}$  be the all-zeros vector, then the above construction reduces to the construction for tridiagonal matrices.
\begin{proposition}\label{penta_decomp}
Let $A$ be a symmetric pentadiagonal 
matrix. 
If the $v_i$ as defined in Equation~\eqref{veqn_pent} are well-defined, then $A=\sum_{i=-1}^nv_iv_i^T$. If the entries  for each $v_i$ are non-negative numbers, then $A$ is completely positive.
\end{proposition}
\begin{proof}
The proof is similar to the tridiagonal case. Consider a symmetric pentadiagonal   $n\times n$ matrix  $A$ such that the vectors in Equation~\eqref{veqn_pent} are well-defined. Let $V_i=v_iv_i^T$ for $i=-1,\dots,n$ and $\tilde{A}=\sum_{i=-1}^nV_i$. We wish to show that $\tilde{A}=A$. From the definition of the $v_i$ given in Equation~\eqref{veqn_pent}, each $V_i$ is symmetric and pentadiagonal with only up to nine nonzero entries and so $\tilde{A}$ itself is symmetric and pentadiagonal.

Now, consider a component $\tilde{a}_{j,j+1}$ of $\tilde{A}$, where $j=1,2,\dots,n-1$. The only $V_i$ that have a nonzero entry in the $(j,j+1)$-th component are $V_{j-1}$ and $V_j$  as $v_{j-1}$ and $v_j$ are the only vectors with both the $j$ and $(j+1)$-th components being nonzero. The $(j,j+1)$-th component of $V_{j-1}+V_j$ is $(v_{j-1})_j(v_{j-1})_{j+1}+(v_j)_j(v_j)_{j+1}$ which, after simplifying,  is in fact $b_j$ and so $\tilde{a}_{j,j+1}=b_j$. By symmetry, we also have $\tilde{a}_{j+1,j}=b_{j}$.

Now, consider a component $\tilde{a}_{j,j+2}$ of $\tilde{A}$, where $j=1,2,\dots,n-2$. The only $V_i$ that has  a nonzero entry in the $(j,j+2)$-th component is $V_{j}$ as $v_{j}$ and $v_{j+2}$ are the only vectors with both the $j$ and $(j+2)$-th components being nonzero. The value  of $(v_j)_j$ is the same as the denominator of $(v_j)_{j+2}$,  and so we simply obtain $\tilde{a}_{j,j+2}=c_j$. By symmetry, we also have $\tilde{a}_{j+2,j}=c_j$.

Now consider a component  on the diagonal of $\tilde{A}$:  $\tilde{a}_{jj}$, where $j=1,2,\dots,n$. The only $V_i$ that have nonzero entries in the $(j,j)$-th component are $V_{j-2}$, $V_{j-1}$, and  $V_j$, and the sum of the  respective values is precisely  $\tilde{a}_{jj}=a_j$ for $j=1,2,\dots,n$. Therefore $A=\tilde{A}=\sum_{i=-1}^nv_iv_i^T$; If the entries for each $v_i$ are non-negative numbers,
then $A$ is completely positive.
\end{proof}

When using Equation~\eqref{veqn_pent} to find a decomposition of a pentadiagonal matrix, it is  simplest to choose the initial vectors $v_{-1}$ and $v_0$ to both be the zero vector. However, Example~\ref{ex:nonzero_init_conds} shows that it is sometimes necessary to choose nonzero initial conditions in order to prove that the given matrix is completely positive.

\begin{example}\label{ex:nonzero_init_conds}
Consider the matrix
\[A=\begin{pmatrix}
 3/4 & 1/8 & 1/8 & 0 & 0 \\
 1/8 & 3/4 & 0 & 1/8 & 0 \\
 1/8 & 0 & 1/2 & 13/40 & 1/20 \\
 0 & 1/8 & 13/40 & 1/2 & 1/20 \\
 0 & 0 & 1/20 & 1/20 & 9/10 \\
\end{pmatrix}.\]
Using Equation~\eqref{veqn_pent} with initial vectors $v_{-1}$ and $v_0$ both taken to be the zero vector, we compute the matrix $V$ such that $A=VV^T$ to be
\[V=\begin{pmatrix}
0 & 0 & \frac{\sqrt{3}}{2} & 0 & 0 & 0 & 0 \\
 0 & 0 & \frac{1}{4 \sqrt{3}} & \frac{\sqrt{\frac{35}{3}}}{4} & 0 & 0 & 0 \\
 0 & 0 & \frac{1}{4 \sqrt{3}} & -\frac{1}{4 \sqrt{105}} & \frac{\sqrt{\frac{67}{35}}}{2} & 0 & 0 \\
 0 & 0 & 0 & \frac{\sqrt{\frac{3}{35}}}{2} & \frac{23}{\sqrt{2345}} & \frac{\sqrt{\frac{339}{335}}}{2} & 0 \\
 0 & 0 & 0 & 0 & \frac{\sqrt{\frac{7}{335}}}{2} & \frac{7 \sqrt{\frac{3}{37855}}}{2} & \sqrt{\frac{101}{113}} \\
\end{pmatrix}.\]
We note that the component $(v_2)_3$ is negative and hence this decomposition cannot be used to prove that $A$ is completely positive. 
It is not surprising that taking the initial conditions to be all zero does not work: if both $v_{-1}$ and $v_0$ are zero vectors, i.e.\ $a_{-1}=a_0=b_0=0,$ then $(v_2)_3>0$ is equivalent to $a_1b_2\geq b_1c_1.$ But in $A,$ $b_1=c_1=1/8$ while $b_2=0,$ so  $a_1b_2< b_1c_1.$

If we instead use the initial conditions $v_{-1}=\begin{pmatrix} 0&0& \dots & 0\end{pmatrix}^T$ and $v_{0}=\begin{pmatrix} \frac 1 2&\frac 1 4&0& \dots & 0\end{pmatrix}^T$, we obtain the decomposition
 $A=WW^T$, where \[W=\begin{pmatrix}
0 & \frac{1}{2} & \frac{1}{\sqrt{2}} & 0 & 0 & 0 & 0 \\
 0 & \frac{1}{4} & 0 & \frac{\sqrt{11}}{4} & 0 & 0 & 0 \\
 0 & 0 & \frac{1}{4 \sqrt{2}} & 0 & \frac{\sqrt{\frac{15}{2}}}{4} & 0 & 0 \\
 0 & 0 & 0 & \frac{1}{2 \sqrt{11}} & \frac{13}{5 \sqrt{30}} & \frac{\sqrt{\frac{4157}{165}}}{10} & 0 \\
 0 & 0 & 0 & 0 & \frac{\sqrt{\frac{2}{15}}}{5} & \frac{23 \sqrt{\frac{11}{62355}}}{10} & \frac{\sqrt{\frac{14861}{4157}}}{2} \\
\end{pmatrix}.\]
 This decomposition shows that $A$ is in fact completely positive.
\end{example}

\begin{example}\label{ex:pentaBlockDiag}
As an analogue to Example~\ref{triBlockEx} consider the matrix
\[\begin{pmatrix}
1/2 & 1/4 & 1/4 & 0 & 0 & 0 \\
 1/4 & 1/2 & 1/4 & 0 & 0 & 0 \\
 1/4 & 1/4 & 1/2 & 0 & 0 & 0 \\
 0 & 0 & 0 & 1/2 & 1/2 & 0 \\
 0 & 0 & 0 & 1/2 & 1/2 & 0 \\
 0 & 0 & 0 & 0 & 0 & 1 \\
\end{pmatrix}.\]
For this matrix, the construction we outline through Equation~\eqref{veqn_pent} never gives a well-defined decomposition regardless of the choice of initial conditions. To see why this is the case, note that $c_2=b_3=c_3=0$. Here we may assume that we have chosen initial conditions such that $(v_1)_1$, $(v_2)_2$, and $(v_3)_3$ are nonzero (otherwise the decomposition would already be ill-defined). From this we immediately obtain
\begin{eqnarray*}
\begin{tabular}{l}
$(v_2)_4=\frac{c_2}{(v_2)_2}=0$\\[0.5cm]
$(v_3)_4=\frac{b_3-(v_2)_4(v_2)_3}{(v_3)_3}=0$\\[0.5cm]
$(v_3)_5=\frac{c_3}{(v_3)_3}=0$.\\[0.5cm]
\end{tabular}
\end{eqnarray*}
Therefore we can compute the following components of $v_4$ to be:
\begin{eqnarray*}
\begin{tabular}{l}
$(v_4)_4=\sqrt{a_4-((v_{3})_4)^2+((v_{2})_4)^2}=\sqrt{a_4}=\frac{1}{\sqrt{2}}$\\[0.5cm]
$(v_4)_5=\frac{b_4-(v_3)_5(v_3)_4}{(v_4)_4}=\frac{b_4}{\sqrt{a_4}}=\frac{1}{\sqrt{2}}$\\[0.5cm]
$(v_4)_6=\frac{c_4}{(v_4)_4}=0$.\\[0.5cm]
\end{tabular}
\end{eqnarray*}
Now, all six of the components that have been calculated so far are completely independent of the initial conditions (except for the requirement that all previous components were well-defined). Therefore the vector $v_5$ is independent of the initial conditions. We then find that
\[(v_5)_5=\sqrt{a_5-(((v_4)_5)^2 + ((v_3)_5)^2)}=\sqrt{\frac{1}{2}-\left(\left(\frac{1}{\sqrt{2}}\right)^2+0\right)}=0.\]
Hence $(v_5)_6$ is not  well-defined.

Similar to Example~\ref{triBlockEx}, we can still make use of our construction to prove that $A$ is completely positive by considering $A$ as the block diagonal matrix
\[A=\begin{pmatrix}
A_1 & 0_{3,2} & 0_{3,1}\\
0_{2,3} & A_2 & 0_{2,1}\\
0_{1,3} & 0_{1,2} &A_3\\
\end{pmatrix}\]
where
\[A_1=\begin{pmatrix}
1/2 & 1/4 & 1/4 \\
 1/4 & 1/2 & 1/4 \\
 1/4 & 1/4 & 1/2\\
\end{pmatrix},\quad
A_2=\begin{pmatrix}
 1/2 & 1/2 \\
 1/2 & 1/2
\end{pmatrix},\quad
A_3=\begin{pmatrix}
1
\end{pmatrix}
\]
From here we can find a decomposition for the three matrices $A_1$, $A_2$, and $A_3$ separately. We find that
\[V_1=\begin{pmatrix}
 0 & 0 & \frac{1}{\sqrt{2}} & 0 & 0 \\
 0 & 0 & \frac{1}{2 \sqrt{2}} & \frac{\sqrt{\frac{3}{2}}}{2} & 0 \\
 0 & 0 & \frac{1}{2 \sqrt{2}} & \frac{1}{2 \sqrt{6}} & \frac{1}{\sqrt{3}} \\
\end{pmatrix}, \quad
V_2=\begin{pmatrix}
 0&\frac{1}{\sqrt{2}} & 0 \\
 0&\frac{1}{\sqrt{2}} & 0 \\
\end{pmatrix}, \quad
V_3=\begin{pmatrix}1\end{pmatrix}\]
where $A_1=V_1V_1^T$, $A_2=V_2V_2^T$, and $A_3=V_3V_3^T$. A decomposition for $A$ can then be formed  by creating the block diagonal matrix
\[V=\begin{pmatrix}
 V_1 & 0_{3,3} & 0_{3,1}\\
0_{2,5} & V_2 & 0_{2,1}\\
0_{1,5} & 0_{1,3} &V_3\\
\end{pmatrix}=
\begin{pmatrix}
 0 & 0 & \frac{1}{\sqrt{2}} & 0 & 0 &0&0&0&0\\
 0 & 0 &\frac{1}{2 \sqrt{2}} & \frac{\sqrt{\frac{3}{2}}}{2} &0& 0&0&0&0 \\
 0 & 0 & \frac{1}{2 \sqrt{2}} & \frac{1}{2 \sqrt{6}} & \frac{1}{\sqrt{3}}&0&0&0&0 \\
 0&0&0&0&0&0&\frac{1}{\sqrt{2}} & 0&0 \\
 0&0&0&0&0&0&\frac{1}{\sqrt{2}} & 0&0 \\
0&0&0&0&0&0&0&0&1\end{pmatrix}\]
and noting  that $A=VV^T$. This proves that  $A$ is completely positive.

Similar to Example~\ref{5x5}, we note that for any matrix $M$,  if   $M$ has columns consisting entirely of zeros these columns can be removed from the matrix $M$ without changing the value of $MM^T$. Therefore we can simplify $V$ to be the $6\times 5$ matrix below, rather than the $6\times 9$ matrix above;
\[V=\begin{pmatrix}
 \frac{1}{\sqrt{2}} & 0 & 0 &0&0\\
 \frac{1}{2 \sqrt{2}} & \frac{\sqrt{\frac{3}{2}}}{2} & 0&0&0 \\
 \frac{1}{2 \sqrt{2}} & \frac{1}{2 \sqrt{6}} & \frac{1}{\sqrt{3}}&0&0 \\
 0&0&0&\frac{1}{\sqrt{2}} & 0 \\
 0&0&0&\frac{1}{\sqrt{2}} &0 \\
0&0&0&0&1\end{pmatrix}\]
\end{example}

We leave a result analogous to Proposition~\ref{PD_CP} in the setting of symmetric pentadiagonal doubly stochastic matrices as an open problem. Example~\ref{ex:nonzero_init_conds} shows that there is a connection between elements in $A$ and how should one choose $v_{-1}$ and $v_0,$ however it is not immediately clear in general. Consider the matrix $A'$ which is equal to $A$ except for the following entries:
\begin{eqnarray*}
{a}'_{11}&=&a_1-(a_0^2+a_{-1}^2)\\
{a}'_{22}&=&a_2-b_0^2\\
{a}'_{21}&=&{a}'_{12}=b_1-b_0a_0.
\end{eqnarray*}
Note that $A={A}'$ provided that $a_0=a_{-1}=0$ and $b_0=0.$ If $A'$ is positive definite, then all of its leading principal minors are positive. However, this does not appear to be enough to conclude that $A$ is completely positive in this setting. Indeed, Equation~\eqref{veqn_pent} yields  $$(v_2)_3=\frac{b_2-\frac{c_1(b_1-b_0a_0)}{a_1-(a_0^2+a_{-1}^2)}}{\sqrt{a_2-\left(\frac{(b_1-b_0a_0)^2}{a_1-(a_0^2+a_{-1}^2)}+b_0^2\right)}},$$ and $(v_2)_3>0$ is equivalent to $$b_2-\frac{c_1(b_1-b_0a_0)}{a_1-(a_0^2+a_{-1}^2)}>0$$ (assuming the denominator of $(v_2)_3$ is well-defined). This expression is equivalent to requiring that the  $3\times 3$ leading principal submatrix of $A'$   with the last row and second last column removed, has positive  determinant.

In general, requiring that $(v_i)_{i+1}>0$, assuming the denominator is well-defined, is equivalent to requiring that the  $(i+1)\times (i+1)$ leading principal submatrix of $A'$   with the last row and second last column removed, has positive determinant.

\subsection{Alternate Construction}
As Example~\ref{ex:nonzero_init_conds} illustrates, there can be some trial and error when it comes to finding a decomposition with all components being positive. Selecting initial conditions that can achieve this may be difficult or even impossible in certain cases. One workaround to this is in the case where the given matrix is block diagonal, as in    Example~\ref{ex:pentaBlockDiag}.

Another technique one can use if decomposing $A$ directly as described by Equation~(\ref{veqn}) or (\ref{veqn_pent}) does not yield results, is described in this Section. It can be used when the given matrix is not necessarily block diagonal. The main idea is to find matrices $\tilde{A}$ and $\hat{A}$ such that $A=\tilde{A}+\hat{A}$ and then decompose $\tilde{A}$ and $\hat{A}$ using Equation~(\ref{veqn}) or (\ref{veqn_pent}).  If $\tilde{A}$ and $\hat{A}$ are completely positive with decompositions $\tilde{A}=VV^T$ and $\hat{A}=WW^T$, then $A$ has a decomposition given by the  matrix $\begin{pmatrix}V & W\end{pmatrix}$, which is simply the matrix constructed with the columns of $V$ followed by the columns of $W$. Below, we provide a  construction that gives $A$ as a sum of two specified positive semidefinite matrices $\tilde{A}$ and $\hat{A}$ that can often be convenient to consider, but in general there are other matrices that work.

Let $A$ be a $n\times n$ symmetric pentadiagonal doubly stochastic matrix. Recall the convention that $b_0=b_{n}=c_{-1}=c_{0}=c_{n-1}=c_{n}=0$. Define the $n\times n$ matrix $\tilde{A}$ to be the matrix with components $\tilde{a}_{ii}=\frac{1}{2}-b_i-b_{i-1}$ for $i\in\{1,\dots,n\}$, $\tilde{a}_{i,i+2}=\tilde{a}_{i+2,i}=c_i$, and all other components being zero.  We  similarly define the $n\times n$ matrix $\hat{A}$ to be the matrix with components $\hat{a}_{ii}=\frac{1}{2}-c_i-c_{i-2}$ for $i\in\{1,\dots,n\}$, $\hat{a}_{i,i+1}=\hat{a}_{i+1,i}=b_i$, and all other components being zero. 
We find that $\tilde{A}+\hat{A}=A$, as desired.

If $A$ is diagonally dominant, $\tilde{A}$ and $\hat{A}$ are diagonally dominant as well, and hence also positive semidefinite. 
As $\tilde{A}$ and $\hat{A}$ are much simpler than $A$,  finding decompositions for both $\tilde{A}$ and $\hat{A}$ with all positive components is often much simpler (if it is possible),  as the next example shows.

\begin{example}
Consider the matrix
\[A=\begin{pmatrix}
 7/12&1/3&1/12&0\\
 1/3&7/12&1/156&1/13\\
 1/12&1/156&7/12&17/52\\
 0&1/13&17/52&31/52
\end{pmatrix}.\]
Since the matrix has order $4$ and is doubly non-negative it must be completely positive by \cite{Berman88}. However, if we try to decompose $A$ using the all zero vectors as our initial conditions, we obtain
\[\begin{pmatrix}
  0 & 0 & \frac{\sqrt{\frac{7}{3}}}{2} & 0 & 0 & 0 \\
 0 & 0 & \frac{2}{\sqrt{21}} & \frac{\sqrt{\frac{11}{7}}}{2} & 0 & 0 \\
 0 & 0 & \frac{1}{2 \sqrt{21}} & -\frac{15}{26 \sqrt{77}} & \frac{\sqrt{\frac{4217}{11}}}{26} & 0 \\
 0 & 0 & 0 & \frac{2 \sqrt{\frac{7}{11}}}{13} & \frac{2491}{26 \sqrt{46387}} & 4 \sqrt{\frac{101}{4217}} \\
\end{pmatrix}.\]
Note the single negative entry. We can try using different initial conditions, but taking a guess-and-check approach is not an ideal strategy. Instead,  now consider the matrices $\tilde{A}$ and $\hat{A}$:
\[\tilde{A}=\begin{pmatrix}
 1/6 & 0 & 1/12 & 0 \\
 0 & 25/156 & 0 & 1/13 \\
 1/12 & 0 & 1/6 & 0 \\
 0 & 1/13 & 0 & 9/52 \\
\end{pmatrix},\quad
\hat{A}=\begin{pmatrix}
 5/12 & 1/3 & 0 & 0 \\
 1/3 & 11/26 & 1/156 & 0 \\
 0 & 1/156 & 5/12 & 17/52 \\
 0 & 0 & 17/52 & 11/26 \\
\end{pmatrix}\]
Decomposing both of these we obtain
\[V=\begin{pmatrix}
  0 & 0 & \frac{1}{\sqrt{6}} & 0 & 0 & 0 \\
 0 & 0 & 0 & \frac{5}{2 \sqrt{39}} & 0 & 0 \\
 0 & 0 & \frac{1}{2 \sqrt{6}} & 0 & \frac{1}{2 \sqrt{2}} & 0 \\
 0 & 0 & 0 & \frac{2 \sqrt{\frac{3}{13}}}{5} & 0 & \frac{\sqrt{\frac{177}{13}}}{10} \\
\end{pmatrix},\quad
W=\begin{pmatrix}
  0 & 0 & \frac{\sqrt{\frac{5}{3}}}{2} & 0 & 0 & 0 \\
 0 & 0 & \frac{2}{\sqrt{15}} & \sqrt{\frac{61}{390}} & 0 & 0 \\
 0 & 0 & 0 & \frac{\sqrt{\frac{5}{4758}}}{2} & \frac{5 \sqrt{\frac{317}{4758}}}{2} & 0 \\
 0 & 0 & 0 & 0 & \frac{17 \sqrt{\frac{183}{8242}}}{5} & \frac{2 \sqrt{\frac{4286}{4121}}}{5} \\
\end{pmatrix}\]
where $\tilde{A}=VV^T$ and $\hat{A}=WW^T$. One can check that $A=\begin{pmatrix}V & W\end{pmatrix}\begin{pmatrix}V & W\end{pmatrix}^T$, where we set $\begin{pmatrix}V & W\end{pmatrix}$ to be (we deleted unnecessary all-zero columns):
\begin{equation}
\begin{pmatrix}V & W\end{pmatrix}=\begin{pmatrix}
   \frac{1}{\sqrt{6}} & 0 & 0 & 0 & \frac{\sqrt{\frac{5}{3}}}{2} & 0 & 0 & 0 \\
  0 & \frac{5}{2 \sqrt{39}} & 0 & 0 & \frac{2}{\sqrt{15}} & \sqrt{\frac{61}{390}} & 0 & 0 \\
  \frac{1}{2 \sqrt{6}} & 0 & \frac{1}{2 \sqrt{2}} & 0&0 & \frac{\sqrt{\frac{5}{4758}}}{2} & \frac{5 \sqrt{\frac{317}{4758}}}{2} & 0 \\
  0 & \frac{2 \sqrt{\frac{3}{13}}}{5} & 0 & \frac{\sqrt{\frac{177}{13}}}{10} & 0 & 0 & \frac{17 \sqrt{\frac{183}{8242}}}{5} & \frac{2 \sqrt{\frac{4286}{4121}}}{5} \\
\end{pmatrix}
\end{equation}
This shows that $A$ is completely positive.
\end{example}
\section{Relation to Cholesky Decomposition}
 The Cholesky factorization $A=LL^T$ decomposes a positive semidefinite matrix $A$ into the product of a lower triangular matrix $L$ and its transpose. 
 
    The Cholesky decomposition for positive definite matrices is unique, as long as the diagonal entries are chosen positive. That is, although the Cholesky algorithm gives a unique factorization $LL^T$ of a positive definite matrix $A$,  where $L$ has positive diagonal, in general, all off-diagonal entries of $L$ are not necessarily nonnegative.  Our algorithm returns a lower triangular matrix as long as $a_0$ is chosen to be zero. Hence, if the $a_0=0$ case works successfully then the result must be precisely the Cholesky decomposition of the matrix (i.e., using the Cholesky algorithm will also  prove complete postivity). In other words, the Cholesky algorithm always gives a nonnegative lower triangular matrix $L$ for positive definite tridiagonal doubly stochastic matrices, which can be seen via our algorithm, and it is unique. 
    
    In the positive semidefinite case, a Cholesky decomposition still exists but is not unique. Our decomposition would be one of the many possible Cholesky factorizations. 
    
    A Cholesky factorization algorithm for  banded matrices is given in \cite[Section 4.3.5]{Golub}. The algorithm therein requires $n(p^2 + 3p)$ \emph{flops}, defined as additions/subtractions/division/multiplication within a matrix computation, and $n$ square roots, where $n$ is the dimension of the matrix and $p$ is the band length. In our setting, $p=1$ and  our algorithm (Equation (\ref{veqn})) requires $3n$ flop: one subtraction and one multiplication in $a_i - (v_{i-1})_i^2$, and one division in $b_i/(v_i)_i)$, for each vector $v_1, \dots, v_n$, and $n$ square roots: $\sqrt{a_i - (v_{i-1})_i^2}$ for each vector $v_1, \dots, v_n$, so it is more efficient, although we are not making claims to having the most efficient possible algorithm.

    If one removes the first two zero columns of the matrix $V$ given in Example~\ref{ex:nonzero_init_conds}, the resulting matrix   is the Cholesky factorization for the given matrix $A$. This example is quite notable then as it gives an example where the Cholesky factorization fails to prove complete positivity but the construction we provide does. This appears to be related  to the fact that the Cholesky factorization provides matrices of cp-rank $n$, whereas in our work, we may choose to look at decompositions with cp-rank $n+1$ for the tridiagonal case, or higher in the pentadiagonal case,  depending on if we set the initial variables $a_0$, $a_{-1}$, etc.\ to be non-zero.

    \section{Conflict of interest}
    The authors have no conflict of interest to report. 
    
    \section{Data availability}
    Data sharing is not applicable to this article as no datasets were generated or analysed during the current study.
    
   \section*{Acknowledgements}
 We thank the referee and editor for their useful comments. S.P.~is supported by NSERC Discovery Grant number 1174582, the Canada Foundation for Innovation (CFI) grant number 35711, and the Canada Research Chairs (CRC) Program grant number 231250. S.P.~would like to thank Rajesh Pereira for helpful discussions early in the project. 

\end{document}